\documentclass{amsart}

\usepackage{ae}
\usepackage[all]{xy}
\usepackage{graphicx,amsfonts,amssymb,amsmath}
\usepackage{eurofont}
\usepackage{natbib}

\usepackage[latin1]{inputenc}
\usepackage[T1]{fontenc}
\usepackage[english]{babel}
\usepackage[cyr]{aeguill}

\newcommand{\chapeau}{{\rlap{\smash{\hbox{\lower4pt\hbox{\hskip1pt$\widehat{\phantom{u}}$}}}}}\mbox{ }}
\usepackage[OT2,T1]{fontenc}
\DeclareSymbolFont{cyrletters}{OT2}{wncyr}{m}{n}
\DeclareMathSymbol{\sha}{\mathalpha}{cyrletters}{"58}
\input cyracc.def

 \newtheorem{thm}{Theorem}[section]
 \newtheorem*{thm*}{Theorem}

 \newtheorem*{pb*}{\textit{Open question}}
 \newtheorem*{cor*}{\textit{Corollary}}

 \newtheorem{lem}[thm]{Lemma}
 \newtheorem{ex}[thm]{Example}
 \newtheorem{prop}[thm]{Proposition}
 \theoremstyle{definition}
 
 \theoremstyle{remark}
 
 \theoremstyle{remark}
 \newtheorem{rem}[thm]{Remark}
 \theoremstyle{remark}

 \numberwithin{equation}{subsection}

 \newcommand{\To}{\longrightarrow}

\usepackage[colorlinks]{hyperref}
\hypersetup{colorlinks,linkcolor=blue,citecolor=blue,}

\begin{document}

\title[]
{The local-global exact sequence for Chow groups of zero-cycles}

\author{ Yongqi LIANG  }

\address{Yongqi LIANG \newline
Département de Mathématiques, \newline Bâtiment 425,\newline Université  Paris-sud 11,\newline  F-91405 Orsay,\newline
 France}

\email{yongqi.liang@math.u-psud.fr}

\thanks{\textit{Key words} : zero-cycles , local-global principle,
Brauer\textendash Manin obstruction}

\thanks{\textit{MSC 2010} : 14G25 (11G35)}

\date{\today.}



\maketitle

\begin{abstract}
A local-global sequence for Chow groups of zero-cycles involving Brauer groups
has been conjectured to be exact for all proper smooth algebraic varieties.
We apply existing methods to construct
several new families of varieties verifying the exact sequence. The examples
are explicit, they are normic bundles over the projective space.
\end{abstract}


\section{Introduction}

Let $k$ be a number field, denote by $\Omega_k$ the set of places of $k.$
Let $X$ be a proper smooth geometrically integral scheme of finite type over $k.$
We write $X_K=X\times_kK$ for any field extension $K$ of $k.$ We write simply
$X_v$ instead of $X_{k_v}$ where $k_v$ is the completion of $k$ at $v\in\Omega_k.$

We consider the local-global principle for zero-cycles on $X.$
Colliot-Th\'el\`ene extended the Brauer\textendash Manin pairing to the product of local Chow groups of
zero-cycles \cite{CT95},
$$\prod_{v\in\Omega_k}\textup{CH}_0(X_v)\times \textup{Br}(X)\to\mathbb{Q}/\mathbb{Z},$$
where $\textup{Br}(-)=\textup{H}_{\scriptsize\mbox{\'et}}^2(-,\mathbb{G}_m)$ is the Brauer group.
We defined the modified Chow group $\textup{CH}'_0(X_v)$ to be the usual Chow group for each non-archimedean
place $v;$ otherwise we set it to be the cokernel of the
norm map $$\textup{CH}'_0(X_v)=\textup{coker}[N_{\bar{k}_v|k_v}:\textup{CH}_0(X_v\times_{k_v}\bar{k}_v)\to \textup{CH}_0(X_v)].$$
In particular, it is $0$ if $v$ is a complex place. The pairing factorizes through the product
of the modified Chow groups. Let us denote $M^\chapeau=\varprojlim_nM/nM$ for any Abelian group $M.$
We get the local-global sequence
$$\textup{CH}_0(X)^\chapeau\mbox{ }\to\prod_{v\in\Omega_k}\textup{CH}'_0(X_v)^\chapeau\mbox{ }\to \textup{Hom}(\textup{Br}(X),\mathbb{Q}/\mathbb{Z}).\leqno(E)$$
Similarly, we denote $\textup{A}_0(X)=\textup{ker}[deg:\textup{CH}_0(X)\to\mathbb{Z}]$
and get a sequence
$$\textup{A}_0(X)^\chapeau\mbox{ }\to\prod_{v\in\Omega_k}\textup{A}_0(X_v)^\chapeau\mbox{ }\to \textup{Hom}(\textup{Br}(X)/\textup{Br}(k),\mathbb{Q}/\mathbb{Z}).\leqno(E_0)$$
We remark that the elements coming from $\textup{Br}(k)$ give no contribution to the pairing for zero-cycles of degree $0.$
The sequence $(E),$ \emph{a fortiori} $(E_0)$ \cite[Remarque 1.1]{Wittenberg}, is conjectured to be exact for all proper smooth
varieties \cite{CTSansuc81} \cite{KatoSaito86} \cite{CT95}.
The exactness of these sequences signifies that the failure of the local-global principle for zero-cycles
is controlled by the Brauer\textendash Manin obstruction.

We will prove the following two theorems. The first theorem is a slightly stronger form of the
main result of \cite{Liang4}. And the second theorem is proved by the restriction-corestriction argument.

\begin{thm*}[Theorem \ref{stronger version}]
Let $X$ be a rationally connected variety defined over $k$ and $L$ be a finite extension of $k.$
Suppose that, for all finite extensions $K$ of $k$ that are linearly disjoint from $L,$
the Brauer\textendash Manin obstruction is the only obstruction to weak approximation for rational points on $X_K.$

Then the sequence $(E)$ is exact for $X.$
\end{thm*}

\begin{thm*}[Theorem \ref{main}]
Let $L/k$ be a finitely generated field extension, in which $k$ is algebraically closed.
Let $\{L_i\}$ be a finite family of finite extensions of $L$ such that the degrees $[L_i:L]$ have no comme
factor.

If for each $i$ the sequence $(E)$ is exact for a proper smooth variety with function field $L_i,$
then $(E)$ is also exact for proper smooth varieties with function field $L.$
\end{thm*}

Finally, we deduce from these theorems the exactness of $(E)$ for several families of
normic bundles over the projective space. The arithmetic of rational points on such bundles has been widely
studied by many authors, \emph{cf.} \cite[Introduction]{B-HB}. In this paper, we obtain results for
zero-cycles on more general normic bundles.

\section{Exactness of the local-global sequence}\label{the sequence}

\subsection{A stronger form of a recent result}\ \\
Concerning a family of varieties studied recently by Derenthal\textendash Smeets\textendash Wei \cite{D-S-W},
we need the following stronger form of the main result of \cite{Liang4} in order to deduce
the exactness of the local-global sequence $(E)$ for such varieties.

\begin{thm}\label{stronger version}
Let $X$ be a proper smooth geometrically integral variety defined over a number field
$k,$ and suppose that $X$ is geometrically rationally connected. Fix a finite extension $L$ of $k.$

Assume that for all finite extensions $K/k$ that are linearly disjoint from $L,$
the Brauer\textendash Manin obstruction is the only obstruction to weak approximation
for $K$-rational points on $X_K$ (\emph{cf.} \cite[\S 5]{Skbook} for definition).

Then the sequence $(E)$ is exact for $X.$
\end{thm}

\begin{proof}
Comparing to the main result Theorems A and B of \cite{Liang4}, the assumption here is slightly weaker.
We consider only finite extensions $K/k$ which are linearly disjoint from a fixed extension $L/k$
instead of considering all finite extensions of $k.$
We are going to explain how the whole proof works with minor modifications under this weaker assumption.
In fact, only two subtitle places need to be modified slightly. First, in the proof
\cite[Thm. 3.2.1]{Liang4} we need to replace everywhere $k'$ by
its compositum $k'_{\scriptsize{\mbox{new}}}=Lk'$ with $L.$ Once the residual field $k(\theta)$ is linearly disjoint from
$k'_{\scriptsize{\mbox{new}}}$ it is \emph{a fortiori} linearly disjoint from $L$ and from $k',$
the weaker assumption applies on $\pi^{-1}(\theta)\simeq X_{k(\theta)}$ for those $\theta$ such that $k(\theta)$
is linearly disjoint from $k'_{\scriptsize{\mbox{new}}},$
thus the proof still works.
Second, in the proof of \cite[Prop. 3.4.1]{Liang4}, we need to make the same replacement.
\end{proof}

\subsection{The restriction-corestriction argument and the exactness of $(E)$}\label{general things}\ \\
Above all we remark that the exactness of $(E)$ is a birationally invariant property for proper smooth
geometrically integral varieties $X$ defined over a number field $k.$
In fact, firstly the Brauer group $\textup{Br}(X)$ is identified with the unramified Brauer group
$\textup{Br}_{\scriptsize\mbox{nr}}(k(X)/k)$ which depends only on the function field $k(X).$
Secondly, the Chow group of zero-cycles $\textup{CH}_0(X)$ is also a birational invariant \cite[Prop. 6.3]{CT-Coray},
and so are the modified Chow groups $\textup{CH}'_0(X_v).$

Hereafter, we denote by $X$ a proper smooth geometrically integral variety defined over a number field $k.$
Then $k$ is algebraically closed in the function field $k(X)$ of $X.$
Consider a finite extension $L$ of $k(X),$ let us denote by $l$ the algebraic closure of $k$ in $L.$
Then there exists a dense open $U$ of $X$ and a integral $k$-variety $V$ with function field $k(V)=L$
satisfying the following commutative diagram
\SelectTips{eu}{12}$$\xymatrix@C=10pt @R=7pt{
V\ar[rr]\ar[ddr]&&U\\
&&\\
&U_l\ar[ruu]&
}$$
where $V\to U$ is a dominant $k$-morphism inducing the extension $L/k(X)$ and
$V\to U_l$ is an $l$-morphism. Thanks to Nagata's compactification, there exists
a proper $l$-morphism $Y\to X_l$ extending $V\to U_l\subset X_l$
in such a way that $V$ is identified with an dense open of an integral $l$-variety $Y.$
Moreover we may assume that $Y$ is regular by Hironaka's resolution of singularities.
The variety $Y$ is geometrically integral over $l$ since $l$ is algebraically closed in $L$ by construction.
For each finite extension $L$ of $k(X)$ we fix such an $l$-variety $Y$ and denote it by $^LY.$
We will consider the exactness of the sequence $(E)$ for the proper smooth geometrically integral $l$-variety
$^LY,$ and it depends only on the function field $L.$

\begin{prop}\label{res-cores}
There exist restriction and corestriction homomorphisms such that the following diagram commutes.
Moreover, the composition of the vertical arrows are identified with the multiplication
by $d=[L:k(X)].$
\SelectTips{eu}{12}$$\xymatrix@C=20pt @R=14pt{
\textup{CH}_0(X)^\chapeau\mbox{ }\ar[r]\ar[d]^{\textup{res}_{\tiny\mbox{\textup{gl}}}}&\prod_{v\in\Omega_k}\textup{CH}'_0(X_v)^\chapeau\mbox{ }\ar[r]\ar[d]^{\textup{res}_{\tiny\mbox{\textup{loc}}}}&\textup{Hom}(\textup{Br}(X),\mathbb{Q}/\mathbb{Z})\ar[d]^{\textup{res}_{\tiny\mbox{\textup{Br}}}}\\
\textup{CH}_0(^LY)^\chapeau\mbox{ }\ar[r]\ar[d]^{\textup{cores}_{\tiny\mbox{\textup{gl}}}}&\prod_{w\in\Omega_l}\textup{CH}'_0(^LY_w)^\chapeau\mbox{ }\ar[r]\ar[d]^{\textup{cores}_{\tiny\mbox{\textup{loc}}}}&\textup{Hom}(\textup{Br}(^LY),\mathbb{Q}/\mathbb{Z})\ar[d]^{\textup{cores}_{\tiny\mbox{\textup{Br}}}}\\
\textup{CH}_0(X)^\chapeau\mbox{ }\ar[r]&\prod_{v\in\Omega_k}\textup{CH}'_0(X_v)^\chapeau\mbox{ }\ar[r]&\textup{Hom}(\textup{Br}(X),\mathbb{Q}/\mathbb{Z})
}$$
Note that the rows are the sequences $(E)$ for $X$ and $^LY.$
\end{prop}

We postpone the proof of this proposition and present immediately its consequence.

\begin{thm}\label{main}
Let $X$ be a proper smooth geometrically integral $k$-variety and $L_i(1\leqslant i\leqslant n)$ be
finite extensions of $k(X)$ of degree $d_i.$ Suppose that $gcd(d_i|1\leqslant i\leqslant n)=1.$
If the sequence $(E)$ is exact for each $l_i$-variety $^{L_i}Y$ then it is exact
for $X.$
\end{thm}

\begin{proof}
Without any difficulty, we can reduce to the case where $n=2.$ By hypothesis, there exists
integers $\alpha$ and $\beta$ such that $\alpha d_1+\beta d_2=1.$
Let $\{z_v\}_v$ be an element of $\prod_{v\in\Omega_k}\textup{CH}'_0(X_v)^\chapeau\mbox{ }$
sending to $0$ in $\textup{Hom}(\textup{Br}(X),\mathbb{Q}/\mathbb{Z}).$ Applying the restriction homomorphism,
its image in $\prod_{w\in\Omega_{l_i}}\textup{CH}'_0(^{L_i}Y_w)^\chapeau\mbox{ }$ comes from an element $z_i$
of $\textup{CH}_0(^{L_i}Y)^\chapeau\mbox{ }$ by exactness of $(E)$ for $^{L_i}Y.$ By Proposition \ref{res-cores}, the element
$\textup{cores}_{\tiny\mbox{gl}}(z_i)\in \textup{CH}_0(X)^\chapeau\mbox{ }$ maps to $d_i\{z_v\}_v\in\prod_{v\in\Omega_k}\textup{CH}'_0(X_v)^\chapeau\mbox{ }.$
Hence the element
$z=\alpha \textup{cores}_{\tiny\mbox{gl}}(z_1)+\beta \textup{cores}_{\tiny\mbox{gl}}(z_2)\in \textup{CH}_0(X)^\chapeau\mbox{ }$ maps to
$\{z_v\}_v\in\prod_{v\in\Omega_k}\textup{CH}'_0(X_v)^\chapeau\mbox{ }.$
The exactness of $(E)$ is thus verified for $X.$
\end{proof}

\begin{proof}[Proof of the Proposition \ref{res-cores}]
The proof is formal.
The morphisms $^LY\to X_l\to X$ correspond to field extensions $L/l(X)/k(X).$
It suffices to verify the statement for the extensions $L/l(X)$ and $l(X)/k(X).$

Consider the $l$-morphism of geometrically integral $l$-varieties $^LY\to X_l$ corresponding to
$L/l(X).$ The restriction and corestriction maps are well defined between the Chow groups of
zero-cycles and between the Brauer groups, moreover their compositions are the
multiplication by $[L:l(X)]$ map. These induce the desired
homomorphisms between the modified Chow groups and we pass to the quotient by $n$ and to the projective
limit $\varprojlim_n,$ which give the desired commutative diagram for the finite morphism $^LY\to X_l.$

Consider the morphism $X_l\to X$ corresponding to $l(X)/k(X).$
There exists also restriction and corestriction homomorphisms between global (resp. local) Chow groups
satisfying a similar commutative diagram to the one in the statement,
moreover their composition
$\textup{CH}_0(X)\to \textup{CH}_0(X_l)\to \textup{CH}_0(X)$ (resp. $\textup{CH}_0(X_v)\to\prod_{w|v}\textup{CH}_0(X_{l,w})\to \textup{CH}_0(X_v)$)
is the multiplication by $[l:k]=[l(X):k(X)]$ map.
It suffices to verify that these homomorphisms factorize
through the modified Chow groups satisfying the following commutative diagram, and moreover $\textup{cores}'_v\circ \textup{res}'_v$
is the multiplication by $[l:k],$ for each archimedean place $v\in\Omega_k.$
\SelectTips{eu}{12}$$\xymatrix@C=20pt @R=14pt{
\textup{CH}_0(X_v)\ar[r]\ar[d]^{\textup{res}_v}&\textup{CH}'_0(X_v)\ar[d]^{\textup{res}'_v}\\
\prod_{w|v}\textup{CH}_0(X_{l,w})\ar[r]\ar[d]^{\textup{cores}_v}&\prod_{w|v}\textup{CH}'_0(X_{l,w})\ar[d]^{\textup{cores}'_v}\\
\textup{CH}_0(X_v)\ar[r]&\textup{CH}'_0(X_v)
}$$
If $v$ is a complex place, it is clear since the modified Chow groups are defined to be $0.$
Hereafter suppose that $v$ is a real place of $k$ and there are $r$ real places $w_i(1\leqslant i\leqslant r)$ of $l$
above $v$ and $s$ complex places $w_i(r+1\leqslant i\leqslant r+s)$ above $v,$ then $r+2s=d=[l:k].$
The lower square is commutative since it is induced by taking cokernels of the following commutative diagram.
\SelectTips{eu}{12}$$\xymatrix@C=40pt @R=14pt{
\prod_{i=1}^{r+s}\textup{CH}_0(X_{l,w_i}\times\bar{l}_{w_i})\ar[d]^{\textup{cores}_v}\ar[r]^{N_{\bar{l}_{w_i}|l_{w_i}}}&\prod_{i=1}^{r+s}\textup{CH}_0(X_{l,w_i})\ar[d]^{\textup{cores}_v}\\
\textup{CH}_0(X_v\times\bar{k}_v)\ar[r]^{N_{\bar{k}_v|k_v}}&\textup{CH}_0(X_v)
}$$
By grouping the real and complex places, we rewrite the upper diagram as follows
\SelectTips{eu}{12}$$\xymatrix@C=20pt @R=14pt{
\textup{CH}_0(X_v)\ar[r]\ar[d]^{\textup{res}_v}&\textup{CH}'_0(X_v)\ar[d]^{\textup{res}'_v}\\
\prod_{i=1}^r\textup{CH}_0(X_{l,w_i})\times\prod_{i=r+1}^{r+s}\textup{CH}_0(X_{l,w_i})\ar[r]&\prod_{i=1}^r\textup{CH}'_0(X_{l,w_i})\times\prod_{i=r+1}^{r+s}0.\\
}$$
We define the homomorphism $\textup{res}'_v$ to be the identity map for the first $r$ components and to be $0$ for the last $s$ components,
the diagram is thus commutative. The composition $\textup{cores}'_v\circ \textup{res}'_v$ is thus the multiplication by $r.$
For any real places $v,$ we have $\textup{CH}'_0(X_v)\simeq(\mathbb{Z}/2\mathbb{Z})^m$
where $m$ is the number of connected components (real topology)
of $X_v(\mathbb{R}),$ \emph{cf.} \cite[Prop 3.1]{CTIschebeck}.
As $r\equiv d\pmod{2},$ the composition $\textup{cores}'_v\circ \textup{res}'_v$ is the multiplication
by $d,$ which completes the proof.
\end{proof}

\section{Applications to normic bundles over the projective space}\label{new families}

In this section, we apply the theorems in \S \ref{the sequence} to establish the exactness of $(E)$
for several families of normic bundles over the projective space.

Let $K/k$ be a finite extension of number fields of degree $m.$ Let $P(t_1,\ldots,t_n)\in k(t_1,\ldots,t_n)$
be a rational function, its poles are situated outside an open subset $U$ of $\mathbb{A}^n.$
Consider an affine variety defined in $\mathbb{A}^{m}\times U$ by the equation
$$N_{K/k}(\textbf{x})=P(t_1,\ldots,t_n).$$
Birationally, it can be viewed as a normic bundle over $\mathbb{P}^n$ via the parameters $t_1,\ldots,t_n.$
One finds in \cite{B-HB} a list of recent results on the arithmetic of rational points on such varieties.
While for zero-cycles, we are going to establish the exactness of $(E)$ for proper smooth
models of these normic bundles.

\subsection{Examples without restriction on the degeneracy loci}
First of all, we obtain some results for certain extensions $K,$
no restriction on $P(t_i,\ldots,t_n)$ is made. Geometrically speaking, we allow arbitrary degeneracy loci of
the normic bundles.
Results in this subsection are all deduced from Theorem \ref{main}.

\begin{ex}\label{3}
Let $P(t_1,\ldots,t_n)\in k(t_1,\ldots,t_n)$ be a rational function.
Let $X$ be a proper smooth variety birationally equivalent to the variety defined by the
following normic equation
$$N_{K/k}(\mathbf{x})=P(t_1,\ldots,t_n),$$ where $K$ is a prime degree $p$ extension of $k.$
Then the sequence $(E)$ is exact for $X.$
\end{ex}

\begin{proof}
The argument dates back to Wei \cite[Thm. 4.3]{Wei}, where the
Hasse principle for zero-cycles of degree $1$ on such varieties was considered.

Firstly, as the left hand side splits into linear factors over $K,$ the variety $X_K$ is $K$-rational and hence
the sequence $(E)$ is exact for $X_K.$

Secondly, consider the Galois closure $M$ of $K$ over $k.$ The Galois group $G=Gal(M/k)$ is then a subgroup of the
symmetric group $S_p.$ Let $H$ be the $p$-Sylow subgroup of $G,$ it is cyclic of order $p.$ Let $L$ be its fixed field,
then the extension $L/k$ is Galois of degree prime to $p.$ The variety $X_L$ is birationally a fibration over $\mathbb{P}^n$
defined by $N_{M/L}(\mathbf{x})=P(t_1,\ldots,t_n),$ with $M/L$ a finite cyclic extension.
According to the following Lemma \ref{abelian-split}, codimension $1$ fibers of this fibration are abelian-split.
The sequence $(E)$ is then exact for $X_L,$ \emph{cf.} \cite[Th\'eor\`eme principal]{Liang3} and its
corrigendum where the codimension $1$ condition is explained.

Theorem \ref{main} applied to the field extensions $K(X)$ and $L(X)$ gives the exactness of $(E)$ for $X.$
\end{proof}

\begin{lem}\label{abelian-split}
Let $P(t_1,\ldots,t_n)\in k(t_1,\ldots,t_n)$ be a rational function.
Let $X$ be a smooth compactification of the smooth locus of the fibration over $\mathbb{P}^n$ defined by the affine equation
$$N_{K/k}(\mathbf{x})=P(t_1,\ldots,t_n),$$ where $K$ is a finite Abelian extension of $k.$

Then for each point $\theta\in\mathbb{P}^n$ of codimension $1,$ the fiber $X_\theta$ is
abelian-split, \emph{i.e.} the fiber $X_\theta$ has an irreducible component $Y$ of
multiplicity $1$ such that the algebraic closure of the residual field $k(\theta)$ in
the function field $k(\theta)(Y)$ is a finite Abelian extension of $k(\theta).$
\end{lem}

\begin{proof}
We follow the argument of a similar result in \cite[pages 117-118]{Wittenberg}.
By Cohen's theorem, the local ring of $\mathbb{P}^n$ at $\theta$ is isomorphic to
$k(\theta)[[t]].$ By restricting the fibration to this local ring, it suffices to prove the statement:
\begin{itemize}
\item[]
If $l$ is an extension of $k$ and $V$ a regular proper $l[[t]]$-scheme of which
the generic fiber $V_\eta$ has an open subscheme defined by the equation
$$N_{K/k}(\mathbf{x})=P(t)$$ in variable $\mathbf{x}$ with coefficient in $l((t)),$ then
the special fiber $V_s$ is abelian-split.
\end{itemize}
Let us denote by $l'$ the compositum $K\cdot l,$ it is a finite Abelian extension of $l.$
Note that the left hand side of the normic equation splits into linear factors over $K,$
whence the generic fiber $V_\eta$ possesses a $l'((t))$-rational point. Since $V$ is proper, we find
a $l'[[t]]$-section of $V'=V\times_{Spec(l[[t]])}Spec(l'[[t]])\to Spec(l'[[t]]).$
The intersection multiplicity of this section with the special fiber $V_s\times_ll'$ of $V'$ equals to
$1$ by the projection formula. We obtain a smooth $l'$-rational point of $V_s,$ and hence
$V_s$ is split by the Abelian extension $l'/l.$
\end{proof}

The forthcoming Examples \ref{4}, \ref{5}, and \ref{6} are also geometrically rational varieties,
we study the local-global principle of zero-cycles, but the similar question for rational points is still far from reaching.

\begin{ex}\label{4}
Let $P(t_1,\ldots,t_n)\in k(t_1,\ldots,t_n)$ be a rational function.
Let $m>1$ be an integer such that $gcd(m,\varphi(m))=1,$ where
$\varphi(\cdot)$ is the Euler function. Let $a\in k^*$ be an element such that $a\notin k^{*q}$ for all integer $q$ dividing $m.$
Then for all proper smooth varieties $X$ birationally equivalent to the variety defined by the equation
$$N_{k(\sqrt[m]{a})/k}(\mathbf{x})=P(t_1,\ldots,t_n),$$
the sequence $(E)$ is exact.
\end{ex}

\begin{proof}
Firstly, the variety $X$ is rational over $K=k(\sqrt[m]{a}),$ hence the sequence $(E)$ is exact for $X_K.$
Under the assumption in the statement, the polynomial $X^m-a$ is irreducible over $k,$ we have $[K:k]=m.$

Secondly, let $\zeta_m$ be a primitive $m$-th root of unity. Consider the cyclotomic extension
$L=k(\zeta_m)$ of $k,$ it is of degree $\varphi(m).$ The variety $X_L$ is defined by the equation
$$N_{L(\sqrt[m]{a})/L}(\mathbf{x})=P(t_1,\ldots,t_n).$$ It is birationally equivalent to a fibration
over $\mathbb{P}^n,$ of which the codimension $1$ fibers are abelian-split by Lemma \ref{abelian-split}.
As the extension $L(\sqrt[m]{a})/L$ is cyclic, the smooth closed fibers of the fibration satisfy weak approximation
for rational points.
Whence the sequence $(E)$ is exact for $X_L,$ \emph{cf.} \cite[Th\'eor\`eme principal]{Liang3} and its
corrigendum where the codimension $1$ condition is explained.

As $m$ and $\varphi(m)$ is supposed coprime, we apply Theorem \ref{main} to the field extensions
$K(X)$ and $L(X)$ of $k(X),$ and we deduce the exactness of $(E)$ for $X.$
\end{proof}

\begin{rem}
The condition that $m$ and $\varphi(m)$ are coprime implies that $m$ is square-free.
There are infinitely many integers $m$ satisfying the condition in the statement, for example
we can take $m=3p$ where $p$ is a prime number congruent to $2$ modulo $3.$
\end{rem}

\begin{ex}\label{5}
Let $P(t_1,\ldots,t_n)\in k(t_1,\ldots,t_n)$ be a rational function.
Let $K_i(i=1,\ldots,m)$ be a field extension of $k$ of prime degree $p_i.$ Suppose that the $p_i$'s are distinct.
Denote by $L$ the compositum of the $K_i$'s.
Then for all proper smooth varieties $X$ birationally equivalent to the variety defined by the equation
$$N_{L/k}(\mathbf{x})=P(t_1,\ldots,t_n),$$
the sequence $(E)$ is exact.
\end{ex}

\begin{proof}
It suffices to prove the case where $m=2.$ According to the assumption, the compositum $L$ is isomorphic
to $K_1\otimes_kK_2.$ We write $K_2=k[T]/(f(T))$ with $f(T)\in k[T]$ an irreducible polynomial.
Then $K_2\otimes_kK_2\simeq K_2[T]/(f_1(T))\times\ldots\times K_2/(f_s(T))$ as $K_2$-algebra, where $f_j(T)\in K_2[T]$ are
irreducible polynomials such that $\prod_{j=1}^sf_j(T)=f(T),$ moreover we may assume that
$f_1(T)$ is of degree $1.$ Denote by $K_2^j$ the field extension $K_2[T]/(f_j(T))$ of $K_2,$ in particular
$K_2^1=K_2.$
The variety $X_{K_2}$ is birationally equivalent to a variety defined by the equation
$$N_{L\otimes_kK_2/K_2}(\mathbf{x})=P(t_1,\ldots,t_n).$$
The last equation rewrites as follows
$$N_{K_1\otimes_kK_2/K_2}(\mathbf{x_1})\cdot\prod_{j=2}^s N_{K_1\otimes_kK_2^j/K_2}(\mathbf{x_j})=P(t_1,\ldots,t_n),$$
birational equivalently
$$N_{L/K_2}(\mathbf{x_1})=\frac{P(t_1,\ldots,t_n)}{\prod_{j=2}^s N_{K_1\otimes_kK_2^j/K_2}(\mathbf{x_j})}.$$
Note that $L/K_2$ is a finite extension of prime degree $p_1,$ Example \ref{3} implies that
the sequence $(E)$ is exact for $X_{K_2}.$

Similar argument shows that $(E)$ is also exact for $X_{K_1},$ and Theorem \ref{main}
proves the exactness of $(E)$ for $X.$
\end{proof}

\begin{rem}
The same argument reduces the exactness of $(E)$ for varieties defined by
$N_{K/k}=P(t_1,\ldots,t_n)$ with $K/k$ an Abelian extension to the cases where $K/k$ is a $p$-extension.
\end{rem}

\begin{rem}
If $L/k$ is a Galois extension of degree $pq$ a product of two distinct prime numbers, then
$L$ is of the form stated in Example \ref{5} by Sylow's theorem.

One can also generate more examples in this manner: let $L/k$ be a Galois extension of degree
$p_1p_2p_3$ where $p_1>p_2>p_3$ are prime numbers. Suppose that $p_1\nmid p_2p_3-1,$
that $p_2\nmid p_3-1,$ and that $p_2\nmid p_1p_3-1.$ Under these assumptions, Sylow's theorem shows that
$Gal(L/k)$ has a normal subgroup of order $p_1$ and a normal subgroup of order $p_2.$ We obtain as fixed fields
Galois extensions $L_1/k$ and $L_2/k$ of degree $p_2p_3$ and $p_1p_3.$ By applying Sylow's theorem once again,
we get subextensions $K_1,$ $K_2,$ and $K_3$ of $L/k$ of degree $p_1,$ $p_2,$ and $p_3.$ Then the field $L$ has to
be the compositum of $K_1,$ $K_2,$ and $K_3.$
\end{rem}

In order to construct more examples from known cases, we suppose that $L_i(i=1,2)$ is a function field
of a proper smooth $k$-variety for which $(E)$ is exact.
The field $L_1$ (resp. $L_2$) is a finite extension of a certain rational function field $K_1=k(T_1,\ldots,T_m)$
(resp. $K_2=k(S_1,\ldots,S_n)$) of degree $d_1$ (resp. $d_2$). If $d_1$ and $d_2$ are coprime,
then $(E)$ is exact for proper smooth models of varieties defined by
$$N_{L_1/k(T_1,\ldots,T_m)}(\mathbf{x})=N_{L_2/k(S_1,\ldots,S_n)}(\mathbf{y})\cdot P(t_1,\ldots,t_n)$$
where $P(t_1,\ldots,t_n)\in k(t_1,\ldots,t_n).$
The exactness of $(E)$ still holds if we replace the extension $L_1/k(T_1,\ldots,T_m)$
by a finite extension $L_1/k$ (satisfying the coprime degree condition).

Alternatively, let $L$ be the compositum of
$L_1(S_1\ldots,S_n)$ and $L_2(T_1,\ldots,T_m)$ over $k(T_1,\ldots,T_m,S_1\ldots,S_n).$
We also get the exactness of $(E)$ for proper smooth models of varieties defined by
$$N_{L/k(T_1,\ldots,T_m,S_1\ldots,S_n)}(\mathbf{x})=P(t_1,\ldots,t_n).$$

These can be summarized as the following example.

\begin{ex}\label{6}
Let $L_i(i=1,2)$ be the function field of one of the following $k$-varieties
\begin{enumerate}
  \item[(a)] a proper smooth variety for which $(E)$ is exact,
  \item[(b)] a smooth projective curve such that its Jacobian has finite Tate\textendash Shafarevich group,
  \item[(c$_1$)] a homogeneous space of a connected linear algebraic group with connected stabilizer,
  \item[(c$_2$)] a homogeneous space of a connected semisimple simply connected linear algebraic group with
                Abelian stabilizer.
\end{enumerate}
The field $L_1$ (resp. $L_2$) is a finite extension of a certain rational function field $K_1=k(T_1,\ldots,T_m)$
(resp. $K_2=k(S_1,\ldots,S_n)$) of degree $d_1$ (resp. $d_2$). This extension corresponds to a dominant
morphism of proper smooth varieties $Y_1\to\mathbb{P}^m$ (resp. $Y_2\to\mathbb{P}^n$). Assume further that $d_1$ and $d_2$ are coprime.

Let $X$ be a proper smooth $k$-variety. Suppose that
\begin{enumerate}
\item the variety $X$ is birationally equivalent to
a fibration $X_1\to\mathbb{P}^m,$ whose pull-back via $Y_1\to\mathbb{P}^m$ has $L_1$-rational
generic fiber,
\item the variety $X$ is birationally equivalent to
a fibration over $X_2\to\mathbb{P}^n,$ whose pull-back via $Y_2\to\mathbb{P}^n$ has $L_2$-rational
generic fiber.
\end{enumerate}
Then the sequence $(E)$ is exact for $X.$
\end{ex}

\begin{proof}
First of all, according to \cite[Prop. 3.3]{CT99HP0-cyc} and \cite[Corollary]{Liang4}, (b), (c$_1$), and (c$_2$)
are particular cases of (a).

Since the fibration $X_1\times_{\mathbb{P}^m}Y_1\to Y_1$ has $L_1$-rational generic fiber, it is birationally
equivalent to $\mathbb{P}^{m_0}\times Y_1$ for some integer $m_0.$
The exactness of $(E)$ for $\mathbb{P}^{m_0}\times Y_1$ is deduced from the exactness for $Y_1.$
Therefore $(E)$ is exact for proper smooth varieties with function field $k(X_1\times_{\mathbb{P}^m}Y_1).$
Similarly, $(E)$ is exact for proper smooth varieties with function field $k(X_2\times_{\mathbb{P}^n}Y_2).$
Theorem \ref{main} applied to $k(X_1\times_{\mathbb{P}^m}Y_1)$ and $k(X_2\times_{\mathbb{P}^n}Y_2),$
which are extensions of coprime degrees $d_1$ and $d_2$ over $k(X),$
gives the exactness of $(E)$ for $X.$
\end{proof}

\subsection{Examples deduced from the arithmetic of rational points}
In the last subsection, most of the extensions $K/k$ are related to Abelian extensions.
For more general $K,$ we need the help of results on the arithmetic of rational points.
The base of the nomic bundles has to be the projective line, and we have to make some restriction on $P(t).$
Examples in this subsection are deduced from Theorem \ref{stronger version}.

The first family of varieties was studied recently by Derenthal\textendash Smeets\textendash Wei \cite{D-S-W}
\cite{D-S-W2}.

\begin{ex}\label{1}
Let $Q(t)\in k[t]$ be a quadratic irreducible polynomial. Let $K$ be a degree $4$ extension of $k$ such that
$Q(t)$ splits in $K.$ Then for all proper smooth varieties $X$ birationally equivalent
to the variety defined by the equation $$N_{K/k}(\mathbf{x})=Q(t),$$
the sequence $(E)$ is exact;
and moreover
$$\textup{A}_0(X)^\chapeau\mbox{ }\To\prod_{v\in\Omega_k}\textup{A}_0(X_v)^\chapeau\mbox{ }$$ is surjective.
\end{ex}

\begin{proof}
For any finite extension $k'/k$ linearly disjoint from the field $k[t]/(Q(t))$ over $k,$
the polynomial $Q(t)$ is still irreducible over $k'$ and splits over $K\cdot k'.$
The Brauer\textendash Manin obstruction is the only obstruction to weak approximation for
rational points on $X_{k'}$ \cite[Thm. 1]{D-S-W}. The exactness of $(E)$ is deduced by Theorem \ref{stronger version},
the exactness of $(E_0)$ thus follows after \cite[Rem. 1.1]{Wittenberg}.
Moreover under the hypothesis, we have $\textup{Br}(X)/\textup{Br}(k)=0$ \cite[Thm. 1]{D-S-W2},
which confirms the surjectivity
of the global-local homomorphism on the degree $0$ part.
\end{proof}

The second family of varieties was firstly studied by Heath-Brown\textendash Skorobogatov \cite{HBS}, then by
Colliot-Th\'el\`ene\textendash Harari\textendash Skorobogatov \cite{CTHS}, and recently
by Swarbrick-Jones \cite{Swarbrick-Jones}.

\begin{ex}\label{2}
Let $m,$ $n$ be positive integers and $c\in k^*$ be a constant.
Then for all proper smooth varieties $X$ birationally equivalent
to the variety defined by the equation $\textup N_{K/k}(\mathbf{x})=ct^m(1-t)^n,$ the sequence $(E)$ is exact.
\end{ex}

\begin{proof}
The statement follows directly from \cite[Thm. 1]{Swarbrick-Jones} and Theorem \ref{stronger version}.
\end{proof}

\bigskip


\bibliographystyle{alpha}
\bibliography{mybib1}
\end{document}